\documentclass[12pt,a4paper]{amsart}
\usepackage[T1]{fontenc}
\usepackage{lmodern}
\usepackage{amsmath,amssymb,amsthm,mathrsfs,mathtools}
\usepackage[a4paper,margin=2.4cm]{geometry}
\usepackage{microtype}
\usepackage{cite}
\usepackage[unicode,pdfencoding=auto]{hyperref}

\mathtoolsset{showonlyrefs=true}

\newtheorem{theorem}{Theorem}[section]
\newtheorem{proposition}[theorem]{Proposition}
\newtheorem{lemma}[theorem]{Lemma}

\theoremstyle{definition}
\newtheorem{remark}[theorem]{Remark}
\newtheorem{example}[theorem]{Example}
\numberwithin{equation}{section}

\newcommand{\Ftwo}{\mathbb F_2}
\newcommand{\A}{\mathscr A}

\newcommand{\Skn}{\mathcal S_k(n)}
\DeclareMathOperator{\pr}{pr}

\def\DD{D\kern-.7em\raise0.4ex\hbox{\char '55}\kern.33em}

\title[Counterexamples to the motivic Peterson conjecture]{A nonvanishing criterion and families of counterexamples to the motivic Peterson conjecture}

\author{\DD\d{\u a}ng V\~o Ph\'uc}
\address{Department of Mathematics, FPT University, An Phu Thinh New Urban Area, Quy Nhon, Vietnam}
\email{dangphuc150488@gmail.com}
\thanks{ORCID: \url{https://orcid.org/0000-0002-6885-3996}}


\begin{document}

\begin{abstract}
The motivic Peterson conjecture asks whether the numerical condition forcing classical polynomial indecomposables to vanish has the same consequence in motivic cohomology. Kameko disproved this assertion by constructing a family of non-hit monomials. To determine all linear relations among the projected classes of these monomials, we compute the image of the motivic hit subspace under a projection onto their span: it consists precisely of the combinations whose coefficients sum to zero. Explicit Bockstein preimages establish the reverse inclusion, showing that no further linear relations hold and yielding a nonvanishing criterion. A binary calculation then produces infinitely many counterexamples for each fixed difference of at least three between the number of variables and the exterior degree.
\end{abstract}

\maketitle

\section{Introduction}

Let $P_n=\Ftwo[t_1,\ldots,t_n]$, with $|t_i|=1$, be the graded polynomial algebra over $\Ftwo$, and write $P_n^d$ for its homogeneous component of degree $d\geq0$. Regard $P_n$ as a module over the mod~$2$ Steenrod algebra $\A$ via its usual action, and let $\A^+\subseteq\A$ denote the positive-degree part of $\A$. An element of $P_n$ is \emph{hit} if it belongs to $\A^+(P_n)$, the subspace spanned by the values of elements of $\A^+$. For each $d\geq0$, the Peterson hit problem asks for a basis of
\[
 QP_n^d=P_n^d/\bigl(\A^+(P_n)\cap P_n^d\bigr).
\]
For $d>0$, let $\mu(d)$ be the least number of integers of the form $2^j-1$, with $j\geq1$, whose sum is $d$, and set $\mu(0)=0$. Write $\alpha(d)$ for the number of ones in the binary expansion of $d$. These conventions agree with Walker--Wood \cite[Definitions~1.4.9 and~2.4.2]{WalkerWoodI}. Proposition~2.4.4 in \cite{WalkerWoodI} gives
\begin{equation}\label{eq:mu-alpha}
 \mu(d)>n\quad\Longleftrightarrow\quad\alpha(d+n)>n.
\end{equation}
Wood's theorem \cite{Wood}, which proves Peterson's conjecture \cite{Peterson}, states that $\mu(d)>n$ implies $QP_n^d=0$; see also \cite[Theorem~2.5.5]{WalkerWoodI}. Background on the classical hit problem and its representation-theoretic methods can be found in \cite{WalkerWoodI,WalkerWoodII}.

A natural question in motivic homotopy theory is whether the numerical bound $\mu(d)>n$ still forces Steenrod indecomposables to vanish in the motivic cohomology of classifying spaces. Throughout the main argument the base field is $\mathbb C$, and $V_n=(\mathbb Z/2)^n$. Let $A^{*,*}$ denote the mod~$2$ motivic Steenrod algebra. Voevodsky's construction of motivic operations and the cohomology of $B\mu_2$ \cite[Theorem~6.10 and \S9]{Voevodsky} give
\begin{equation}
 M_n=H^{*,*}(BV_n/\mathbb C;\Ftwo)
 =\frac{\Ftwo[\tau,x_1,\ldots,x_n,y_1,\ldots,y_n]}
 {(x_1^2+\tau y_1,\ldots,x_n^2+\tau y_n)},
\end{equation}
where $|\tau|=(0,1)$, $|x_i|=(1,1)$, and $|y_i|=(2,1)$. Following Kameko \cite[\S1]{KamekoMotivic}, let $A^{*,*}_+$ be the sum of the homogeneous components of bidegree $(p,q)$ with $p+q>0$, and let $A^{*,*}_+(M_n)$ denote the span of their values; in particular, multiplication by $\tau$ is included. Writing $M_n^{d,*}=\bigoplus_q M_n^{d,q}$, we consider the motivic hit quotient
\[
 QM_n^{d,*}=M_n^{d,*}/\bigl(A^{*,*}_+(M_n)\cap M_n^{d,*}\bigr).
\]
In contrast to Wood's theorem  \cite{Wood}, Kameko \cite[Theorem~1.3]{KamekoMotivic} showed that the implication $\mu(d)>n\Rightarrow QM_n^{d,*}=0$ is false.

To analyze $QM_n^{d,*}$, one reduces modulo $\tau$. The ideal $(\tau)$ in $A^{*,*}$ is two-sided, and the ideal $\tau M_n$ is invariant under the motivic Steenrod operations; see \cite[\S11]{Voevodsky} and \cite[\S4]{KamekoMotivic}. Passing to the quotient replaces the relations $x_i^2=\tau y_i$ by $x_i^2=0$, yielding an algebra and an $A^\sharp$-module
\[
 A^\sharp=A^{*,*}/(\tau),\qquad
 N_n=M_n/(\tau)
     =\Lambda(x_1,\ldots,x_n)\otimes\Ftwo[y_1,\ldots,y_n].
\]
Thus $N_n$ decomposes by exterior degree into a direct sum of polynomial modules, a feature absent in the classical ring $P_n$. Letting $A^\sharp_+$ denote the positive-degree part of $A^\sharp$, we set
\[
 QN_n^{d,*}=N_n^{d,*}/\bigl(A^\sharp_+(N_n)\cap N_n^{d,*}\bigr).
\]
As recalled in Section~\ref{sec:preliminaries}, reduction modulo $\tau$ induces an isomorphism $QM_n^{d,*}\cong QN_n^{d,*}$.

Fix integers $n\geq2$ and $1\leq k<n$, and put
\begin{equation}\label{eq:degree}
 d_1=(n-1)(2^k-1),\qquad d=k+2d_1.
\end{equation}
Let $[n]=\{1,\ldots,n\}$ and $\Skn=\{I\subseteq[n]:|I|=k\}$. For $I=\{i_1<\cdots<i_k\}\in\Skn$, define
\begin{equation}\label{eq:zI}
 z_I=
 \left(\prod_{s=1}^k x_{i_s}y_{i_s}^{\,2^k-2^{k-s}-1}\right)
 \left(\prod_{i\notin I}y_i^{\,2^k-1}\right).
\end{equation}
We use the same notation for this monomial in $M_n$ and for its image in $N_n$. These are the increasing translates of Kameko's monomial $z_k=z_{\{1,\ldots,k\}}$. Kameko proved that each $z_I$ is non-hit \cite[Proposition~1.4]{KamekoMotivic}, and the description of the Bockstein image in \cite[Proposition~5.1]{KamekoMotivic} shows that the projection of any hit element onto the span of $\{z_I : I\in\Skn\}$ has an even number of nonzero coordinates, which implies the nonvanishing of odd sums. Two natural questions remained open:\\ 
\textit{Whether any further linear combinations of the classes $[z_I]$ vanish in $QN_n^{d,*}$, and for which parameters $(n,k)$ beyond $k=n-3$ these classes satisfy $\mu(d)>n$}.

In this paper we answer both questions. Let $V$ be the $\Ftwo$-vector space with basis $\{e_I:I\in\Skn\}$, and define the coefficient-sum functional
\[
 \varepsilon:V\longrightarrow\Ftwo,\qquad
 \varepsilon\left(\sum_{I\in\Skn}c_Ie_I\right)
       =\sum_{I\in\Skn}c_I.
\]

\begin{theorem}\label{thm:main}
There is a linear map $\vartheta:N_n^{d,*}\longrightarrow V$, constructed in Section~\ref{sec:projection}, such that $\vartheta(z_I)=e_I$ for every $I\in\Skn$ and
\begin{equation}\label{eq:main}
 \vartheta\bigl(A^\sharp_+(N_n)\cap N_n^{d,*}\bigr)
       =\ker\varepsilon.
\end{equation}
If $u\in N_n^{d,*}$ satisfies $\varepsilon(\vartheta(u))=1$, then its class in $QN_n^{d,*}$ is nonzero, and every lift of $u$ to $M_n^{d,*}$ has nonzero class in $QM_n^{d,*}$. In particular, any sum of an odd number of distinct monomials $z_I$ represents a nonzero class in both motivic quotients.
\end{theorem}

Equation~\eqref{eq:main} shows that the parity of the coefficient sum is the exact linear obstruction carried by the span of Kameko's monomials: every vector in $\ker\varepsilon$ arises as the $\vartheta$-image of an explicit Bockstein boundary, so no additional linear relations hold among the projected classes of the $z_I$. Because $\vartheta$ annihilates the remaining basis summands of $N_n^{d,*}$, the condition $\varepsilon(\vartheta(u))=1$ is a sufficient criterion for nonvanishing rather than a characterization of all non-hit elements in $QN_n^{d,*}$. Proposition~\ref{prop:arithmetic} then applies this criterion to $n=2^r+1$ and $k=n-m$ for every fixed difference $m\geq3$. It determines the exact binary threshold on $r$ for which $\mu(d)>n$, recovering Kameko's subfamily for $m=3$ \cite[Theorem~1.3]{KamekoMotivic} and yielding infinitely many counterexamples to the motivic Peterson conjecture for each $m\geq3$.

The comparison with the classical hit problem uses a different specialization. Complex realization sends
\[
 \tau\longmapsto1,\qquad x_i\longmapsto t_i,\qquad
 y_i\longmapsto t_i^2,
\]
and induces isomorphisms $M_n/(\tau+1)\cong P_n$ and $A^{*,*}/(\tau+1)\cong\A$ \cite[\S1]{KamekoMotivic}. This specialization does not descend to a map from $QM_n$ to $QP_n$ under the positive-degree convention above: $\tau$ is motivically hit, whereas its realization $1$ is not classically hit. Moreover, the relation $x_i^2=y_i$ collapses the exterior-degree grading used to construct $\vartheta$. Thus the motivic nonvanishing results do not contradict Wood's theorem \cite{Wood}; when $\mu(d)>n$, the realized polynomials belong to $\A^+(P_n)$.

\begin{remark}
The proof of Theorem~\ref{thm:main} uses only the presentation of $N_n$ and its operation formulas. These are unchanged over any algebraically closed field $K$ of characteristic different from $2$: the coefficient ring is $\Ftwo[\tau]$ by \cite[Theorem~6.17]{VoevodskyNormResidue} and \cite[Corollary~2.13]{HoyoisKellyOstvaer}, and the presentation and operations follow from \cite[Theorem~6.10 and Proposition~9.7]{Voevodsky} and \cite[Theorem~1.1]{HoyoisKellyOstvaer}. Consequently, Theorem~\ref{thm:main} and Proposition~\ref{prop:arithmetic} in Section \ref{sec:proof} hold over any such field $K$. For an arbitrary field $K$ of characteristic different from $2$, the non-hit conclusions for the displayed polynomials follow by extending scalars to an algebraic closure. Indeed, pullback commutes with the Steenrod operations \cite[\S2.4]{HoyoisKellyOstvaer}, so a hit expression over $K$ would give a hit expression for the same polynomial over its algebraic closure.
\end{remark}

\section{Preliminaries and examples}
\label{sec:preliminaries}

The proof of Theorem~\ref{thm:main} uses Kameko's monomial basis \cite{KamekoMotivic}
to define the projection and his binary-column congruence to identify
the surviving terms of the Bockstein preimages. After fixing the
grading and polynomial quotient, we recall these results and verify
directly the Bockstein containment in \eqref{eq:main} for the full
range $1\leq k<n$. We then construct the projection and illustrate
it for $n=3$.

\subsection{Bigrading and operations}

Write $\Lambda_n=\Lambda(x_1,\ldots,x_n)$ and
$Y_n=\Ftwo[y_1,\ldots,y_n]$. The subspace $\Lambda_n^a$ has
bidegree $(a,a)$, whereas $Y_n^{2b}$ consists of polynomials of
ordinary polynomial degree $b$ and has bidegree $(2b,b)$. Thus
\begin{equation}\label{eq:bigrading}
 \Lambda_n^a\otimes Y_n^{2b}\subseteq N_n^{a+2b,a+b},
 \qquad
 N_n^{d,*}=\bigoplus_{a+2b=d}\Lambda_n^a\otimes Y_n^{2b}.
\end{equation}
Each summand has the fixed motivic weight $a+b$. In particular,
the summand with exterior degree $k$ in degree $d$ from
\eqref{eq:degree} has bidegree $(d,k+d_1)$. When a single degree
is specified below, it is the first, or topological, degree.

The algebra $A^\sharp$ is generated over $\Ftwo$ by the Bockstein
$Q_0$ of bidegree $(1,0)$ and the reduced powers $\mathcal P^s$
of bidegree $(2s,s)$, $s\geq1$. Their action on the generators is
\begin{equation}
 \begin{gathered}
 Q_0(x_i)=y_i,\qquad Q_0(y_i)=0,\qquad
 \mathcal P^s(x_i)=0\quad(s\geq1),\\
 \mathcal P^1(y_i)=y_i^2,\qquad
 \mathcal P^s(y_i)=0\quad(s\geq2).
 \end{gathered}
\end{equation}
Here $Q_0$ is a derivation. The Cartan formula for the reduced powers
in $M_n$ has an additional term divisible by $\tau$; after reduction
to $N_n$ it becomes the usual Cartan formula. Consequently,
\begin{equation}\label{eq:power-tensor}
 \mathcal P^s(\lambda y)=\lambda\mathcal P^s(y)
 \qquad(\lambda\in\Lambda_n,\ y\in Y_n, s\geq1).
\end{equation}
These formulas are recorded in \cite[\S\S1,4]{KamekoMotivic}; see
also \cite[Proposition~9.7 and Lemma~11.3]{Voevodsky}.
Let $P\subseteq A^\sharp$ be the subalgebra generated by the
reduced powers, and let $P^+$ be its positive-degree part.

The quotient map $q_\tau:M_n\to N_n$ is equivariant with
respect to the homomorphism $A^{*,*}\to A^\sharp$. Since
$\tau M_n\subseteq A^{*,*}_+(M_n)$, it induces the isomorphism
\begin{equation}\label{eq:QM-QN}
 QM_n^{d,*}\cong QN_n^{d,*}
\end{equation}
described in \cite[\S4]{KamekoMotivic}. In particular, a polynomial
whose reduction modulo $\tau$ is non-hit cannot be hit in $M_n$.

\subsection{Binary weights and the polynomial quotient}
\label{sec:weights}

For a monomial
$z=x_1^{\epsilon_1}\cdots x_n^{\epsilon_n}y_1^{a_1}\cdots y_n^{a_n}$
in $N_n$, where $\epsilon_i\in\{0,1\}$, define
\[
 \epsilon_i+2a_i=\sum_{j\geq0}\alpha_{ij}(z)2^j,
 \qquad \alpha_{ij}(z)\in\{0,1\},\qquad 1\leq i\leq n.
\]
Set
\[
 \alpha_i(z)=\sum_{j\geq0}\alpha_{ij}(z),\qquad
 \omega_j(z)=\sum_{i=1}^n\alpha_{ij}(z),\qquad
 u_j(z)=(\alpha_{1j}(z),\ldots,\alpha_{nj}(z)).
\]
Thus $\omega_0(z)$ is its exterior degree. For a polynomial
monomial $y_1^{a_1}\cdots y_n^{a_n}$, the entry
$\alpha_{ij}$, $j\geq1$, is the coefficient of $2^{j-1}$ in
$a_i$. This is Kameko's indexing convention
\cite[\S4]{KamekoMotivic}.

With $n,k$ fixed as in \eqref{eq:degree}, let $G_n\subseteq Y_n$
be the subspace spanned by $P^+(Y_n)$ and all monomials $y$ with
\begin{equation}
 (\omega_1(y),\ldots,\omega_k(y))<(n-1,\ldots,n-1)
\end{equation}
in lexicographic order from the left. For the fixed integer $k$, we
call $(\omega_1(y),\ldots,\omega_k(y))$ the \emph{binary weight vector}
of a monomial $y\in Y_n$, or simply its weight. This vector is
distinct from the motivic weight, the second coordinate of the
bidegree. We write
$G_n^{2b}=G_n\cap Y_n^{2b}$; the dependence of $G_n$ on $k$
is understood. In particular,
\begin{equation}\label{eq:power-G}
 P^+(Y_n)\subseteq G_n.
\end{equation}

The ring isomorphism
\begin{equation}\label{eq:psi}
 \psi:Y_n\longrightarrow P_n=\Ftwo[t_1,\ldots,t_n],
 \qquad \psi(y_i)=t_i,
\end{equation}
sends $Y_n^{2b}$ to $P_n^b$ and halves the topological degree.
Its domain has motivic bidegree $(2b,b)$ in this component; its
codomain has only the classical degree $b$. It satisfies
$\psi(\mathcal P^s y)=\operatorname{Sq}^s\psi(y)$
\cite[\S4]{KamekoMotivic}. Thus $\psi$ identifies
$(Y_n/G_n)^{2b}$ with Kameko's classical quotient $(P_n/F_n)^b$,
where $F_n$ is spanned by $\A^+(P_n)$ and the monomials whose
first $k$ binary column sums are lexicographically smaller than
$(n-1,\ldots,n-1)$. The map \eqref{eq:psi} is distinct from
complex realization, which sends $y_i$ to $t_i^2$.

\subsection{Kameko's monomial basis}

For a nonempty subset $J=\{j_1<\cdots<j_\ell\}\subseteq[n]$
with $\ell\leq k$, define
\begin{equation}
 b_J=y_{j_1}^{\,2^{\ell-1}-1}
       \left(\prod_{s=2}^{\ell}y_{j_s}^{\,2^k-2^{\ell-s}-1}\right)
       \left(\prod_{i\notin J}y_i^{\,2^k-1}\right)
       \in Y_n^{2d_1}.
\end{equation}
An empty product is $1$. If $I\in\Skn$, let
$x_I=\prod_{i\in I}x_i$, so that $x_Ib_I=z_I$.
Let
\[
 \pi:\Lambda_n^k\otimes Y_n^{2d_1}
 \longrightarrow\Lambda_n^k\otimes(Y_n/G_n)^{2d_1}
\]
be the quotient map.

\begin{proposition}[Kameko \cite{KamekoMotivic}, Propositions~3.1, 3.2 and 4.1]
\label{prop:basis}
The residue classes of the monomials $b_J$, indexed by
$1\leq |J|\leq k$, form a basis of $(Y_n/G_n)^{2d_1}$.
Every monomial in $Y_n^{2d_1}$ has weight at most
$(n-1,\ldots,n-1)$. A monomial of smaller weight belongs to
$G_n$; a monomial of equal weight is congruent modulo $G_n$
to a unique $b_J$.
Consequently, a basis of
$\Lambda_n^k\otimes(Y_n/G_n)^{2d_1}$ is
\[
 \bigl\{\pi(x_Ib_J):I\in\Skn,\ 1\leq |J|\leq k\bigr\}.
\]
\end{proposition}

The stated basis and uniqueness are proved in
\cite[Propositions~3.1, 3.2 and 4.1, and the argument following
Lemma~3.5]{KamekoMotivic}. We shall also use the following
congruence from \cite[Lemma~3.3]{KamekoMotivic}, transported by
$\psi$: if $y\in Y_n^{2d_1}$ has weight
$(n-1,\ldots,n-1)$, exchanging two neighboring columns among
$u_1(y),\ldots,u_k(y)$ does not change its residue class modulo
$G_n$. The stated congruence is symmetric in the two monomials,
so it can be applied in either direction.

In Kameko's notation \cite[\S4]{KamekoMotivic}, the basis splits as
\begin{equation}
 \begin{split}
 M_0&=\{x_Ib_J:I\in\Skn,\ 1\leq |J|\leq k,\ J\ne I\},\\
 M_1&=\{x_Ib_I:I\in\Skn\}=\{z_I:I\in\Skn\}.
 \end{split}
\end{equation}
Set $U_0=\langle\pi(m):m\in M_0\rangle$. Proposition~\ref{prop:basis}
identifies the space $V$ in the Introduction with
$\langle\pi(z_I):I\in\Skn\rangle$, by $e_I=\pi(z_I)$, and gives
\begin{equation}
 \Lambda_n^k\otimes(Y_n/G_n)^{2d_1}=U_0\oplus V.
\end{equation}
The following result is the Bockstein calculation used in the proof
of Theorem~\ref{thm:main}.

\begin{proposition}[Kameko {\cite[Proposition~5.1]{KamekoMotivic}}]
\label{prop:kameko-Q0}
If $w\in N_n^{d-1,*}$ is a monomial with exterior degree $k+1$,
then $\pi(Q_0w)$ is a linear combination of vectors in
$\pi(M_0)$ and vectors $\pi(z_I)+\pi(z_J)$ with
$I,J\in\Skn$.
\end{proposition}

\begin{proof}
First let $v\in Y_n^{2d_1}$ be a monomial with binary weight vector
$(n-1,\ldots,n-1)$, and define
\[
 Z(v)=\{i\in[n]:\alpha_{ij}(v)=0
                       \text{ for some }1\leq j\leq k\}.
\]
The first $k$ polynomial binary columns contribute
$(n-1)\sum_{j=1}^k2^{j-1}=d_1$ to the ordinary polynomial degree.
All higher columns therefore vanish. Each of these $k$ columns has
exactly one zero, so $1\leq |Z(v)|\leq k$.
We apply the classical reductions in
\cite[Lemmas~3.3 and~3.4, and the proof of Proposition~3.2]{KamekoMotivic}
to $\psi(v)$ and transport the resulting congruences back to $Y_n$.
These reductions preserve $Z(v)$. A neighboring-column exchange
preserves the indices of the zero rows. The other reduction
replaces three zero-row indices $(\gamma,\gamma,\delta)$ by
$(\gamma,\delta,\delta)$, where $\gamma$ and $\delta$ are the
indices of the unique zeros in the two distinct column vectors.
It therefore preserves their set of distinct values.
Since the canonical monomial $b_J$ has zero-row
set exactly $J$, the normal form in Proposition~\ref{prop:basis}
must satisfy
\[
 v\equiv b_{Z(v)}\pmod{G_n}.
\]

For any subset $S\subseteq[n]$, write $x_S=\prod_{i\in S}x_i$.
Express the monomial in the proposition as
\[
 w=x_L\prod_{i=1}^n y_i^{a_i},\qquad
 |L|=k+1,\qquad \sum_{i=1}^n a_i=d_1-1.
\]
For $t\in L$, put
\[
 v_t=y_t^{a_t+1}\prod_{i\ne t}y_i^{a_i},\qquad
 q_t=x_{L\setminus\{t\}}v_t.
\]
Then $Q_0w=\sum_{t\in L}q_t$ and $v_t\in Y_n^{2d_1}$.
By Proposition~\ref{prop:basis}, every $\pi(q_t)$ is either zero
or one canonical basis class. If none is in $\pi(M_1)$, all
nonzero classes are in $\pi(M_0)$, as required.

Otherwise choose $t\in L$ with $\pi(q_t)\in\pi(M_1)$, and
set $I=L\setminus\{t\}$ and $T=2^k-1$.
The polynomial $v_t$ has binary weight $(n-1,\ldots,n-1)$,
and the preceding normal form gives $Z(v_t)=I$.
There are $k$ zeros in total and $k$ rows in $I$, so each of
these rows contains exactly one zero. Thus its exponent is
$T-2^{h_i-1}$, where $i\mapsto h_i$ is a bijection
from $I$ to $\{1,\ldots,k\}$, and every exponent outside $I$
is $T$. Let $b\in I$ be the unique index with $h_b=1$.
Recovering the source exponents gives
\[
 a_t=a_b=T-1,\qquad
 a_c=T-2^{h_c-1}\quad(c\in I\setminus\{b\}),\qquad
 a_i=T\quad(i\notin L).
\]
The polynomial $v_b$ again has binary weight
$(n-1,\ldots,n-1)$, with zero-row set $L\setminus\{b\}$:
its first-column zero is at $t$, and the other columns retain
their zeros. Consequently,
\[
 \pi(q_t)=\pi(z_{L\setminus\{t\}}),\qquad
 \pi(q_b)=\pi(z_{L\setminus\{b\}}).
\]
For $c\in L\setminus\{t,b\}$, the source exponent $a_c$ is
odd, whereas $a_t$ and $a_b$ are even. Incrementing $a_c$
therefore makes the first polynomial column of $v_c$ zero in
exactly these three rows. Its column sum is $n-3<n-1$, so
$\pi(q_c)=0$. This case is empty when $k=1$.
It follows that
\[
 \pi(Q_0w)=\pi(z_{L\setminus\{t\}})
             +\pi(z_{L\setminus\{b\}}).
\]
The two supports are distinct because $b\ne t$. This proves
the assertion in the remaining case for every $1\leq k<n$,
including $k=n-1$.
\end{proof}

\subsection{The projection and small examples}
\label{sec:projection}

Let $\pr_k:N_n^{d,*}\to\Lambda_n^k\otimes Y_n^{2d_1}$ be
the projection in \eqref{eq:bigrading}, and let
$p_{M_1}:U_0\oplus V\to V$ be the projection with kernel $U_0$.
Define
\begin{equation}\label{eq:vartheta}
 \vartheta=p_{M_1}\circ\pi\circ\pr_k.
\end{equation}
Then $\vartheta(z_I)=e_I$. In particular, $\vartheta$ is onto,
and it vanishes on the summands of \eqref{eq:bigrading} whose
exterior degree is different from $k$.

\begin{example}
Let $n=3$ and $k=1$, so that $d_1=2$ and $d=5$.
The relevant component of $G_3$ is
\[
 G_3^4=\langle y_1^2,y_2^2,y_3^2\rangle,
\]
since these are precisely the degree-$4$ polynomial monomials of
weight smaller than $2$, and $\mathcal P^1(y_i)=y_i^2$.
Thus $(Y_3/G_3)^4$ has basis $[y_2y_3],[y_1y_3],[y_1y_2]$.
Here
\[
 z_{\{1\}}=x_1y_2y_3,\qquad
 z_{\{2\}}=x_2y_1y_3,\qquad
 z_{\{3\}}=x_3y_1y_2.
\]
The set $M_0$ consists of the other six monomials
$x_i b_{\{j\}}$ with $i\ne j$.
For instance,
\[
 \vartheta(x_1y_2y_3)=e_{\{1\}},\quad
 \vartheta(x_1y_1y_2)=0,\quad
 \vartheta(x_1y_1^2)=0,\quad
 \vartheta(x_1x_2x_3y_1)=0.
\]
These four values illustrate, respectively, the retained basis
vector, an $M_0$ term, a term whose polynomial factor lies in
$G_3$, and a term of the wrong exterior degree. The first three
monomials have bidegree $(5,3)$, whereas the last has bidegree
$(5,4)$. Finally,
\[
 Q_0(x_1x_2y_3)=x_2y_1y_3+x_1y_2y_3
              =z_{\{2\}}+z_{\{1\}}.
\]
\end{example}

\begin{example}
Let $n=3$ and $k=2$, so that $d_1=6$ and $d=14$.
Writing $z_{ij}=z_{\{i,j\}}$, the three distinguished monomials are
\[
 z_{12}=x_1x_2y_1y_2^2y_3^3,\qquad
 z_{13}=x_1x_3y_1y_2^3y_3^2,\qquad
 z_{23}=x_2x_3y_1^3y_2y_3^2.
\]
For $w=x_1x_2x_3y_1^2y_2y_3^2\in N_3^{13,8}$, the derivation
formula gives
\[
 Q_0w=z_{23}+x_1x_3y_1^2y_2^2y_3^2
               +x_1x_2y_1^2y_2y_3^3.
\]
The polynomial factor of the middle term has weight $(0,3)<(2,2)$.
For the last term, the Cartan formula gives
\[
 \mathcal P^1(y_1y_2y_3^3)
   =y_1^2y_2y_3^3+y_1y_2^2y_3^3+y_1y_2y_3^4.
\]
The polynomial $y_1y_2y_3^4$ has weight $(2,0)<(2,2)$, so
$y_1^2y_2y_3^3\equiv y_1y_2^2y_3^3\pmod{G_3}$.
It follows that
\[
 \pi(Q_0w)=\pi(z_{23})+\pi(z_{12}),\qquad
 \vartheta(Q_0w)=e_{\{2,3\}}+e_{\{1,2\}}.
\]
This calculation exhibits the congruence needed when the two
surviving terms do not already have their polynomial exponents in
the order used in \eqref{eq:zI}.
\end{example}

\section{Proof of the main theorem and counterexamples}
\label{sec:proof}

We now prove Theorem~\ref{thm:main} using the projection constructed
above, and then establish Proposition~\ref{prop:arithmetic}.
The first lemma removes reduced-power contributions. The second
constructs a Bockstein preimage for the sum of two basis vectors
whose indexing subsets differ by one element. Successively replacing
elements of a subset then gives all vectors in $\ker\varepsilon$.

\begin{lemma}\label{lem:reduction}
In degree $d=k+2d_1$,
\[
 \vartheta\bigl(A^\sharp_+(N_n)\cap N_n^{d,*}\bigr)
      =\vartheta\bigl(Q_0(N_n^{d-1,*})\bigr).
\]
\end{lemma}

\begin{proof}
Let $s\geq1$ and $u\in N_n^{d-2s,*}$. By linearity, it is enough
to take $u=\lambda y$, with $\lambda$ an exterior monomial and
$y$ a polynomial monomial. Formula~\eqref{eq:power-tensor}
shows that $\mathcal P^s$ preserves exterior degree. Hence
$\pr_k(\mathcal P^su)=0$ if the exterior degree of $\lambda$
is not $k$. In the remaining case,
\[
 \lambda\in\Lambda_n^k,\qquad y\in Y_n^{2d_1-2s},
\]
where this component is zero if $s>d_1$. For $s\leq d_1$,
\[
 \pr_k(\mathcal P^s u)=\lambda\mathcal P^s(y)
          \in\Lambda_n^k\otimes G_n^{2d_1}
\]
by \eqref{eq:power-G}. The quotient map $\pi$ kills this subspace.
Therefore
\begin{equation}\label{eq:powers-vanish}
 \vartheta(\mathcal P^su)=0\qquad(s\geq1).
\end{equation}

Every positive-degree element of $A^\sharp$ is a linear
combination of nonempty words in $Q_0$ and the reduced powers.
Consider the leftmost operation in such a word acting on an
element of $N_n$. If it is $Q_0$, its value in topological degree
$d$ belongs to $Q_0(N_n^{d-1,*})$. If it is a reduced power,
its image under $\vartheta$ is zero by \eqref{eq:powers-vanish}.
Taking homogeneous components and summing proves one inclusion.
The reverse inclusion follows from $Q_0\in A^\sharp_+$.
\end{proof}

\begin{lemma}\label{lem:preimages}
Let $I,J\in\Skn$ satisfy $|I\cap J|=k-1$. There is a monomial
$w_{I,J}\in N_n^{d-1,k+d_1}$ such that
\begin{equation}\label{eq:pair-image}
 \pi(Q_0w_{I,J})=\pi(z_I)+\pi(z_J).
\end{equation}
In particular, $\vartheta(Q_0w_{I,J})=e_I+e_J$.
\end{lemma}

\begin{proof}
Put $C=I\cap J$, write $I=C\cup\{a\}$ and $J=C\cup\{b\}$,
and let $L=C\cup\{a,b\}$. Choose a bijection
$\phi:C\longrightarrow\{2,\ldots,k\}$; when $k=1$, this is the
unique map between two empty sets. For every $1\leq i\leq n$,
define
\[
 E_i=\begin{cases}
 2^k-2,&i\in\{a,b\},\\
 2^k-1-2^{\phi(i)-1},&i\in C,\\
 2^k-1,&i\notin L,
 \end{cases}
 \qquad
 w_{I,J}=\left(\prod_{i\in L}x_i\right)\prod_{i=1}^ny_i^{E_i}.
\]
All the exponents are nonnegative.

In the notation of Section~\ref{sec:weights}, its digits satisfy,
for $1\leq i\leq n$,
\[
 \alpha_{i0}(w_{I,J})=\begin{cases}1,&i\in L,\\0,&i\notin L,
 \end{cases}
 \qquad
 \alpha_{i1}(w_{I,J})=\begin{cases}0,&i\in\{a,b\},\\1,&i\notin\{a,b\},
 \end{cases}
\]
and, for $2\leq j\leq k$,
\[
 \alpha_{ij}(w_{I,J})=
 \begin{cases}0,&i=\phi^{-1}(j),\\1,&i\ne\phi^{-1}(j).
 \end{cases}
\]
All digits in positions $j>k$ vanish. Each row has exactly one
zero among positions $0,\ldots,k$, so
$\alpha_i(w_{I,J})=k$ for every $i$. The column sums are
\[
 \omega_0(w_{I,J})=k+1,\qquad
 \omega_1(w_{I,J})=n-2,\qquad
 \omega_j(w_{I,J})=n-1\quad(2\leq j\leq k).
\]
Consequently,
\begin{align*}
 |w_{I,J}|
  &=(k+1)+2(n-2)+(n-1)\sum_{j=2}^k2^j\\
  &=(k+1)+2(n-2)+(n-1)(2^{k+1}-4)\\
  &=(n-1)(2^{k+1}-2)+k-1=d-1.
\end{align*}
The sum is empty for $k=1$, when the same equality holds.
Its polynomial factor has degree $2(d_1-1)$, so
\eqref{eq:bigrading} gives motivic weight
$(k+1)+(d_1-1)=k+d_1$.

For $t\in L$, let
\[
 q_t=x_{L\setminus\{t\}}y_t^{E_t+1}
                        \prod_{i\ne t}y_i^{E_i}.
\]
Since $Q_0$ is a derivation, $Q_0w_{I,J}=\sum_{t\in L}q_t$.
If $t\in C$, the digits in positions $0$ and $1$ of
$1+2E_t$ are both $1$. Applying $Q_0$ to $x_t$ adds $1$ to
this combined exponent and changes its position-$1$ digit to $0$.
No other row changes, and hence
$\omega_1(q_t)=n-3<n-1$. The polynomial factor of $q_t$
therefore lies in $G_n$, so
\[
 q_t\in\Lambda_n^k\otimes G_n^{2d_1},\qquad \pi(q_t)=0
 \quad(t\in C).
\]
Only $q_a$ and $q_b$ remain. Their exterior supports are,
respectively, $J=L\setminus\{a\}$ and $I=L\setminus\{b\}$.

We identify their polynomial residue classes explicitly. In $q_a$,
the exponent of $y_a$ is $2^k-1$, that of $y_b$ is $2^k-2$,
and that of $y_c$, for $c\in C$, is
$2^k-1-2^{\phi(c)-1}$. All exponents outside $J$ are thus
$2^k-1$. Among the first $k$ polynomial binary columns, each
column has exactly one zero, and these zeros occur in the distinct
rows indexed by $J$. In particular, this polynomial factor has
weight $(n-1,\ldots,n-1)$.

Write $J=\{j_1<\cdots<j_k\}$. By successively applying the
neighboring-column congruence recalled after
Proposition~\ref{prop:basis}, we may arrange that the zero in
column $h$, for $1\leq h\leq k$, occurs in row $j_{k-h+1}$.
These interchanges preserve the residue class modulo $G_n$.
The exponent in row $j_s$ is then
$2^k-1-2^{k-s}$, while every exponent outside $J$ is $2^k-1$.
The resulting polynomial is precisely $b_J$. Thus
\[
 \pi(q_a)=\pi(x_Jb_J)=\pi(z_J).
\]
The same argument, with $a$ and $b$ interchanged, gives
$\pi(q_b)=\pi(z_I)$. These identifications use Kameko's
column congruence \cite[Lemma~3.3]{KamekoMotivic} and the
unique representatives in \cite[Proposition~4.1]{KamekoMotivic}.
They specify the two terms in the second case of
\cite[Proposition~5.3]{KamekoMotivic}, even when $\phi$ puts
their exponents in a different order from \eqref{eq:zI}.
There is no additional $M_0$ term in these congruences.
Summing the classes of the $q_t$ proves \eqref{eq:pair-image}.
\end{proof}

\begin{proof}[Proof of Theorem~\ref{thm:main}]
Use the map \eqref{eq:vartheta}, which already satisfies
$\vartheta(z_I)=e_I$. By Lemma~\ref{lem:reduction}, it suffices
to compute the image of $Q_0(N_n^{d-1,*})$. If a monomial
$w\in N_n^{d-1,*}$ has exterior degree different from $k+1$,
then $\pr_k(Q_0w)=0$, because $Q_0$ lowers exterior degree by
one. If its exterior degree is $k+1$,
Proposition~\ref{prop:kameko-Q0} expresses $\pi(Q_0w)$ as a
sum of $M_0$ classes and pairs $\pi(z_I)+\pi(z_J)$.
The first type is killed by $p_{M_1}$, and every pair has
coefficient sum zero. By linearity,
\[
 \vartheta\bigl(Q_0(N_n^{d-1,*})\bigr)\subseteq\ker\varepsilon.
\]

For the opposite inclusion, fix $I_*\in\Skn$. If
$I\in\Skn$ and $I\ne I_*$, enumerate
$I_*\setminus I=\{a_1,\ldots,a_h\}$ and
$I\setminus I_*=\{b_1,\ldots,b_h\}$. Define
\[
 I_j=\bigl(I_*\setminus\{a_1,\ldots,a_j\}\bigr)
          \cup\{b_1,\ldots,b_j\},\qquad 0\leq j\leq h.
\]
Then $I_0=I_*$, $I_h=I$, and consecutive sets intersect in
$k-1$ elements. Lemma~\ref{lem:preimages} therefore gives
\[
 \vartheta\left(Q_0\sum_{j=1}^h w_{I_{j-1},I_j}\right)
   =\sum_{j=1}^h(e_{I_{j-1}}+e_{I_j})=e_{I_*}+e_I,
\]
since every intermediate basis vector occurs twice.
The vectors $e_{I_*}+e_I$, $I\ne I_*$, form a basis of
$\ker\varepsilon$: if $\sum_I c_I=0$, then
\[
 \sum_{I\in\Skn}c_Ie_I
     =\sum_{I\ne I_*}c_I(e_{I_*}+e_I),
\]
and independence follows by comparing the coefficients at $e_I$ for
$I\ne I_*$. Every basis vector has the required Bockstein preimage,
so the reverse inclusion holds. This proves \eqref{eq:main}.

If $u\in N_n^{d,*}$ were hit, \eqref{eq:main} would give
$\varepsilon(\vartheta(u))=0$. Therefore the stated condition
implies that its class in $QN_n^{d,*}$ is nonzero.
The isomorphism \eqref{eq:QM-QN} proves the assertion about
every lift to $M_n$. Finally, for a subset $\mathcal T\subseteq\Skn$,
\[
 \varepsilon\left(\vartheta\left(\sum_{I\in\mathcal T}z_I\right)\right)
       =|\mathcal T|\pmod2.
\]
This proves the assertion for odd-cardinality sums.
\end{proof}

The numerical condition for a family of these classes can now be
determined directly. We retain the base field $\mathbb C$ in the
following proposition.

\begin{proposition}\label{prop:arithmetic}
Let $m\geq3$ and $r\geq1$ be integers with $2^r+1>m$, and set
\[
 n=2^r+1,\qquad k=n-m,\qquad
 d=(n-1)(2^{k+1}-2)+k.
\]
Then
\begin{equation}
 \mu(d)>n\quad\Longleftrightarrow\quad
 r\geq m+\alpha(m-3).
\end{equation}
Whenever this condition holds, every sum of an odd number of
distinct monomials $z_I$ represents a nonzero class in
$QM_n^{d,*}$. In particular, each fixed $m\geq3$ gives
infinitely many pairs $(n,d)$ with $\mu(d)>n$ and
$QM_n^{d,*}\ne0$.
\end{proposition}

\begin{proof}
The assumptions give $1\leq k<n$. Substitution of $k=n-m$
and $n-1=2^r$ gives
\begin{align*}
 d+n
   &=(n-1)2^{k+1}-2(n-1)+k+n\\
   &=(n-1)2^{n-m+1}-(m-2)\\
   &=2^{r+n-m+1}-(m-2).
\end{align*}
Put $B=r+n-m+1$ and $X=m-2$. We have
$1\leq X<2^r-1<2^B$. Since
\[
 2^B-X=(2^B-1)-(X-1),
\]
the $B$ binary digits of $2^B-X$ are the complements of
the $B$ binary digits of $X-1$, allowing leading zeros.
Therefore
\begin{equation}
 \alpha(d+n)=B-\alpha(X-1)
            =r+n-m+1-\alpha(m-3).
\end{equation}
Using \eqref{eq:mu-alpha}, we obtain
\[
 \mu(d)>n
 \quad\Longleftrightarrow\quad
 r-m+1-\alpha(m-3)>0
 \quad\Longleftrightarrow\quad
 r\geq m+\alpha(m-3),
\]
where the final equivalence uses integrality. The nonvanishing
assertion follows from Theorem~\ref{thm:main}. For fixed $m$,
every integer $r\geq m+\alpha(m-3)$ also satisfies $2^r+1>m$,
since $r\geq m\geq3$. These infinitely many values of $r$
give distinct values of $n$, proving the last assertion.
\end{proof}

\begin{remark}
For $m=3$, the bound is $r\geq3$. Since
$\alpha(n-2)=\alpha(2^r-1)=r$, this recovers the subfamily
$n=2^r+1$ of Kameko's counterexamples
\cite[Theorem~1.3]{KamekoMotivic}. For $m=4$, the bound is
$r\geq5$. The same proposition gives the exact bound for
each further fixed value of $m$ without a separate calculation.
\end{remark}

\end{document}